\documentclass[12pt,english]{article}

\usepackage{hyperref}
\hypersetup{
    unicode = false,
    pdftoolbar = true,
    pdfmenubar = true,
    pdffitwindow = true,
    pdftitle = {StoMADS},
    pdfauthor = {DZAHINI},
    pdfsubject = {StoMADS},
    pdfnewwindow = true,
    pdfkeywords = {Stochastic, MADS},
    colorlinks = true,
    linkcolor = blue,
    citecolor = blue,
    filecolor = black,
    urlcolor = blue,
    breaklinks = true
}

\usepackage[english]{babel}
\usepackage[T1]{fontenc}
\usepackage[latin1]{inputenc}
\usepackage{amsfonts}
\usepackage{dsfont}
\usepackage{amsmath}
\usepackage{amsthm}
\usepackage{amssymb}
\usepackage{multirow}
\usepackage{epsfig}
\usepackage{soul}
\usepackage{colortbl}
\usepackage{geometry}
\usepackage{times}
\usepackage{graphicx}
\usepackage{mathrsfs}
\usepackage{frcursive}
\usepackage[ruled, vlined, linesnumbered]{algorithm2e}
\usepackage{color}
\usepackage{booktabs}

\definecolor{gris3}{rgb}{0.3,0.3,0.3}

\definecolor{Green}{rgb}{0,.6,0}
\definecolor{Blue}{rgb}{0,0,1}
\definecolor{Red}{rgb}{1,0,0}
\definecolor{Gray}{rgb}{0.2,0.2,0.2}
\definecolor{Maroon}{rgb}{0.6,0.05,0.03}

\newcommand{\blc}{\color{black}}
\newcommand{\bld}{\color{black}}

\newcommand{\blr}{\color{black}}

\usepackage{tikz}
\usetikzlibrary{shapes}
\usetikzlibrary{fit, arrows, calc, positioning}
\usetikzlibrary{arrows, automata}
\usetikzlibrary{positioning}
\usetikzlibrary{shadows}

\tikzstyle{c} = [rectangle, rounded corners, draw]
\tikzset{box/.style={draw,rectangle,rounded  corners=0pt, fill=gray!15, minimum  width=3cm, minimum  height=2cm}}
\tikzset{loz/.style={draw,diamond, aspect=1.5, fill=gray!15}}

\tikzstyle{rnd}=[circle,fill=blue!25,minimum  width=5em]

\newtheorem{theorem}{Theorem}  
\newtheorem{proposition}{Proposition}

\newtheorem{lemma}{Lemma}

\newtheorem{assumption}{Assumption}

\newtheorem{definition}{Definition}

\newcommand{\Z}{\mathbb{Z}}
\newcommand{\R}{\mathbb{R}}
\newcommand{\Esp}{\mathbb{E}}
\newcommand{\rn}{\mathbb{R}^n}

\newcommand{\N}{\mathbb{N}}

\newcommand{\pb}{\mathbb{P}}

\newcommand{\dpl}{\delta^k}
\newcommand{\Dp}{\Delta^k}
\newcommand{\Dk}{\Delta^k}

\newcommand{\dk}{\delta^k}
\newcommand{\epr}{{\epsilon'}}

\newcommand{\fd}{f}

\newcommand{\fok}{f^k_0}
\newcommand{\fsk}{f^k_s}
\newcommand{\dmax}{\delta_{\max}}
\newcommand{\Fok}{F^k_0}
\newcommand{\Fsk}{F^k_s}
\newcommand{\ef}{\varepsilon_f}
   
\newcommand{\isucc}{\mathds{1}_{S}}

\newcommand{\ijk}{\mathds{1}_{J_k}}

\newcommand{\ijkc}{\mathds{1}_{\bar{J_k}}}

\newcommand{\isbar}{\mathds{1}_{\bar{S}}} 

\newcommand{\ite}{\mathds{1}_{\{T_{\epsilon'} > k\}}}
\newcommand{\deps}{\delta_{\epsilon'}}
\newcommand{\dmx}{\delta_{\max}}
\newcommand{\accolade}[1]{\left\lbrace#1\right\rbrace}
\newcommand{\abs}[1]{\left\lvert#1\right\rvert}
\newcommand{\E}[1]{\mathbb{E}\left(#1\right)}
\newcommand{\pr}[1]{\mathbb{P}\left(#1\right)}

\newcommand{\norme}[1]{\left\lVert#1\right\rVert}

\newcommand{\norminf}[1]{{\left\lVert#1\right\rVert}}

\title{
    Expected complexity analysis of stochastic direct-search
}

\author{
	\href{mailto:kwassi-joseph.dzahini@polymtl.ca}{Kwassi Joseph Dzahini}\thanks{GERAD and DÈpartement de MathÈmatiques et de GÈnie Industriel, \'Ecole Polytechnique de MontrÈal, C.P. 6079, Succ. Centre-ville, MontrÈal, QuÈbec H3C 3A7, Canada (\href{mailto:kwassi-joseph.dzahini@polymtl.ca}{Kwassi-Joseph.Dzahini@polymtl.ca}).
	}
}

\begin{document}
\maketitle

\vspace*{-0.5cm}

\noindent
{\bf Abstract:}
This work presents the convergence rate analysis of stochastic variants of the broad class of direct-search methods of directional type. It introduces an algorithm designed to optimize differentiable objective functions $f$ whose values can only be computed through a stochastically noisy blackbox. 
The proposed stochastic directional direct-search (SDDS) algorithm accepts new iterates by imposing a sufficient decrease condition on so called probabilistic estimates of the corresponding unavailable objective function values. The accuracy of such estimates is required to hold with a sufficiently large but fixed probability $\beta$. The analysis of this method utilizes an existing supermartingale-based framework proposed for the convergence rates analysis of stochastic optimization methods that use adaptive step sizes. It aims to show that the expected number of iterations required to drive the norm of the gradient of $f$ below a given threshold $\epsilon$ is bounded in  $\mathcal{O}\left(\epsilon^{\frac{-p}{\min(p-1,1)}}/(2\beta-1)\right)$ with $p>1$. Unlike prior analysis using the same aforementioned framework such as those of stochastic trust-region methods and stochastic line search methods, SDDS does not use any gradient information to find descent directions. However, its convergence rate is similar to those of both latter methods with a dependence on $\epsilon$ that also matches that of the broad class of deterministic directional direct-search methods which accept new iterates by imposing a sufficient decrease condition.

\bigskip 

\noindent
{\bf \bld Keywords:} Blackbox optimization, Derivative-free optimization, Stochastic optimization, Convergence rate, Direct-search, Stochastic processes.

\newpage
\section{Introduction} 
Direct-search methods constitute a broad class of derivative-free optimization (DFO) methods where at each iteration, the DFO algorithm evaluates the objective function at a collection of points and acts solely based on those function values without any model building or derivative approximation~\cite{AuHa2017,CoScVibook}. Such methods include as well those based on simplices like the classical Nelder-Mead method and its numerous variants, as those of directional type where an improvement in the objective function is guaranteed by moving along a direction defined by a better point~\cite{Vicente2013}. 

This work focuses on the convergence rate analysis of stochastic variants of the broad class of directional direct-search methods analyzed in~\cite{Vicente2013}, using a supermartingale-based framework proposed in~\cite{blanchet2019convergence} and elements from~\cite{KoLeTo03a}. It introduces a stochastic directional direct-search (SDDS) algorithm designed for stochastic blackbox optimization (BBO) and aims to solve the following unconstrained stochastic blackbox optimization problem {\bld which often arises in modern statistical machine learning}: 
\begin{equation}\label{problem}
\underset{x\in\rn}{\min}\ f(x)\quad \text{with}\quad f(x)=\Esp_{\Theta}[f_\Theta(x)]
\end{equation}
where $\Theta$ is a real-valued random variable following some unknown distribution, $f_{\Theta}$ denotes the blackbox, the stochastically noisy computable version of the objective function $f:\rn\to\R$ which is numerically unavailable, and $\Esp_{\Theta}$ denotes the expectation with respect to $\Theta$. {\bld Note that $\Theta$ is considered as a data point for many machine learning problems~\cite{paquette2018stochastic}.}

Significant theoretical and algorithmic advances have been made in the field of stochastic DFO in the recent years with the aim of solving {\bld P}roblem~\eqref{problem}. Thus, numerous algorithms have been developed, 
most of which carry out either an estimation of the gradient of $f$ using a single simulation, or a processing of the simulation model as a blackbox. However, since the simulation model can be inaccessible in many real applications, or the gradient can be too expensive to estimate computationally, direct-search optimization methods ``appear to be the most promising option''~\cite{audet2019stomads}. 

Several recent works have proposed directional direct-search algorithms with full supported convergence rates analysis. Vicente~\cite{Vicente2013} proved that to drive the gradient of an objective function below a threshold $\epsilon\in(0,1)$, the number of iterations required by the broad class of directional direct-search methods that use a sufficient decrease condition when accepting new iterates, is bounded in~$\mathcal{O}\left(\epsilon^{\frac{-p}{\min(p-1,1)}}\right)$, with $p>1$. Directional direct-search methods based on {\it probabilistic descent}, that incorporate random gradient, was recently proposed and analyzed by Gratton et al.~\cite{GrRoViZh2015} with worst-case complexity, and global rates results. However, both aforementioned previous works assume that the objective function is deterministic, i.e, function values are exactly computed. 

Audet et al.~\cite{audet2019stomads} recently proposed StoMADS, a stochastic variant of the mesh adaptive direct-search (MADS) algorithm~\cite{AuDe2006}, with full-supported convergence analysis based on Clarke calculus and martingale theory. Using an algorithmic framework similar to that of MADS, Alarie et al.~\cite{alarie2019optimization} also proposed another variant of MADS capable to optimize noisy blackboxes corrupted with Gaussian noise, and proved convergence results of the proposed method using statistical inference techniques. Nevertheless, no convergence rates analysis have been carried out for both methods. 


The main novelty of the present work is that unlike many prior research on convergence rate analysis of stochastic DFO methods (see for example~\cite{berahas2019global,blanchet2019convergence,paquette2018stochastic,wang2019stochastic} and references therein), especially those on stochastic trust-region~\cite{blanchet2019convergence} and line search~\cite{paquette2018stochastic} methods, SDDS does not use any first-order information to find descent directions. Instead, such directions are provided by a positive spanning set and are chosen in such a way to ensure that they never become close to loosing the positive spanning property. 
However, as emphasized in~\cite{Vicente2013}, ``{\it it is not unreasonable}'' to expect that SDDS shares a similar worst case complexity bound of the latter methods in term of the expected number of iterations. Indeed, one of the directions of any positive spanning set makes an acute angle with the negative gradient, provided that the objective function is continuously differentiable~\cite{KoLeTo03a,Vicente2013}. This latter remark is in fact the cornerstone of the analysis in the present manuscript. Moreover, unlike the deterministic framework of~\cite{Vicente2013}, the proposed method accepts new iterates by imposing a sufficient decrease condition on so called {\it probabilistic estimates} of the corresponding unavailable objective function values, which accuracy is required to hold with a sufficiently large but fixed probability $\beta>1/2$. 
However, even though such probability $\beta$ of encountering sufficiently accurate estimates is not required to equal one, SDDS is shown to have desirable convergence properties. Specifically, as main theoretical result of the present work, the expected number of iterations required by SDDS to drive the gradient of $f$ below a threshold $\epsilon$ is shown to be bounded in $\mathcal{O}\left(\epsilon^{\frac{-p}{\min(p-1,1)}}/(2\beta-1)\right)$, using a supermartingale-based framework proposed in~\cite{blanchet2019convergence}.
Moreover, a  subsequence of random iterates generated by SDDS is shown to drive the norm of the gradient of $f$ to zero with probability one. Note also that the analysis is made very general in the present manuscript in the sense that it is not limited to $p=2$, compared to several similar works, but instead, extends to $p>1$.
To the best of our knowledge, this research is the first to propose a convergence rate analysis of a stochastic direct-search algorithm of directional type.

This manuscript is organized as follows. Section~\ref{ExpSect2} introduces an outline of the proposed stochastic algorithm and requirements on so-called  probabilistic estimates that guarantee convergence at an appropriate rate. It is followed by Section~\ref{ExpSect3} which presents a general framework of a stochastic process that is required to carry out the convergence rate analysis in Section~\ref{ExpSect4}. Section~\ref{ExpSect4} {\bld also presents} a $\liminf$-type first-order convergence result for SDDS, followed by a discussion and suggestions for future work.

\section{The SDDS method and probabilistic estimates}\label{ExpSect2}
SDDS, the stochastic algorithm analyzed in the present manuscript, is a direct-search method that uses inexact or noisy information about the objective $f$, specifically making use of so called probabilistic estimates. This section introduces the general framework of SDDS and discusses the requirements on the probabilistic estimates that guarantee the convergence of the algorithm. 

\subsection{The stochastic directional direct-search algorithm}\label{ExpSub21}
The stochastic directional direct-search methods under study in the present manuscript use an algorithmic framework similar to that of the broad class of methods analyzed in~\cite{Vicente2013}, i.e., a framework that can describe the main features of generating set search (GSS)~\cite{KoLeTo03a}, pattern search and generalized pattern search (GPS)~\cite{AuDe03a}. 

Each iteration of a directional direct-search method is composed of two main steps: the SEARCH step which is optional and the POLL step on which relies the convergence analysis. For simplicity of presentation, Algorithm~\ref{algomads} does not show any SEARCH step. During the POLL, trial points are generated in a subset $\mathcal{P}^k=\{x^k+\dk d: d\in\mathbb{D}^k\}$ of the space of variables, where $x^k$ denotes the incumbent solution, $\dk$ the step size and $\mathbb{D}^k$ is a positive spanning set~\cite{AuHa2017,CoScVibook}. Thus, the POLL step which follows stricter rules, consists of a local exploration of the variables space, unlike the SEARCH step which consists of a global exploration. 

In Algorithm~\ref{algomads}, since objective function values $f(x)$ are unavailable,  $\fok$ and $\fsk$ denote respectively the estimates of $f(x^k)$ and $f(x^k+s^k)$, with $s^k=\dk d$, constructed making use of evaluations of the noisy objective $f_{\Theta}$. In order for the information provided by $\fok$ and $\fsk$ to determine the iteration type, i.e., successful or unsuccessful, both estimates are required to be $\varepsilon_f$-accurate, with $\ef>0$, according to the following definition similar to those in~\cite{audet2019stomads,blanchet2019convergence,chen2018stochastic,paquette2018stochastic}.
\begin{definition}\label{accurate} Let ${\rho}:(0,+\infty)\rightarrow(0,+\infty)$ be a continuous and non-decreasing function satisfying ${\rho}(t)/t\rightarrow 0$ when $t\searrow 0$. 
	$f^k$ is called $\ef$-accurate estimate of $f(x^k)$ for a given $\dk$ if \[\abs{f^k-f(x^k)}\leq\ef {\rho}(\dk).\]
\end{definition}

Following the terminology in~\cite{KoLeTo03a}, the function $\rho$ in Definition~\ref{accurate} represents the ``{\it forcing function}''. Sufficient information to determine the iteration type is provided next.
\begin{proposition}\label{prop1}
	Let $\fok$ and $\fsk$ be $\ef$-accurate estimates of $f(x^k)$ and $f(x^k+s^k)$ respectively, and let $\gamma>2$ be a fixed constant. Then the followings hold:
	\begin{eqnarray}
	\text{if}& &\fsk-\fok\leq -\gamma\ef{\rho}(\dk),\ \ \text{then}\ \ f(x^k+s^k)-f(x^k)\leq -(\gamma-2)\ef{\rho}(\dk):=u^k_{\textnormal{sto}}\label{successDecrease}\\
	\text{if}& &\fsk-\fok> -\gamma\ef{\rho}(\dk),\ \ \text{then}\ \ f(x^k+s^k)-f(x^k)> -(\gamma+2)\ef{\rho}(\dk):=\ell^k_{\textnormal{sto}}\label{failure}
	\end{eqnarray}
\end{proposition}
\begin{proof}
	The proof straightforwardly follows from Definition~\ref{accurate} and the equality 
	\[f(x^k+s^k)-f(x^k)=f(x^k+s^k)-\fsk+(\fsk-\fok)+\fok-f(x^k). \]
\end{proof}
In addition to the results in Proposition~\ref{prop1}, the definition of the iteration type in Algorithm~\ref{algomads} is motivated by the following remarks. First, notice that in the stochastic framework of StoMADS where $\rho(t)=t^2$, since as in~\eqref{failure} the inequality $\fsk-\fok> -\gamma\ef{\rho}(\delta^k_p)$ ($\delta^k_p$ denoting the so-called {\it frame} size parameter) does not necessarily lead to an increase in the unavailable objective function $f$, two types of unsuccessful iterations have been distinguished. Unsuccessful iterations which are called {\it certain}, are characterized by $\fsk-\fok\geq -\gamma\ef{\rho}(\delta^k_p)$ and lead to an increase in $f$ whenever both estimates $\fok$ and $\fsk$ are accurate, while those such that $-\gamma\ef{\rho}(\delta^k_p)<\fsk-\fok<\gamma\ef{\rho}(\delta^k_p)$ are called {\it uncertain} since they lead to $-(\gamma+2)\ef{\rho}(\dk)<f(x^k+s^k)-f(x^k)<(\gamma+2)\ef{\rho}(\dk)$. Then, even though updating the frame size parameter according to $\delta_p^{k+1}=\tau\delta^k_p$ on {\it uncertain} unsuccessful iterations, and  $\delta_p^{k+1}=\tau^2\delta^k_p$ whenever the unsuccessful iteration is {\it certain} ($\tau\in (0,1)$ being a rational number), the corresponding sequence $\{\delta_p^k\}_{k\in\N}$ was shown in~\cite{audet2019stomads} to converge to zero. Note also that this kind of update is the only one that differentiates {\it certain} iterations from those that are {\it uncertain}. In the present work, the step size parameter $\dk$ is therefore updated on unsuccessful iterations according to $\delta^{k+1}=\tau\dk$, where $\tau$ is a real number in $(0,1)$. As a consequence, {\it certain} unsuccessful iterations will not be differentiated from {\it uncertain} ones. In other words, every iteration such that $\fsk-\fok> -\gamma\ef{\rho}(\dk)$ will be called unsuccessful. 

However, let put an emphasis on the specific choice of $\tau$ by means of the following additional remarks. Note that in the general deterministic framework described in~\cite{Vicente2013}, the amount of decrease in the objective function on successful iterations is such that $f(x^k+s^k)-f(x^k)\leq -\rho(\dk):=u^k_{\text{det}}$ while unsuccessful iterations are characterized by $f(x^k+s^k)-f(x^k)> -\rho(\dk):=\ell^k_{\text{det}}$. Thus, the equality $\ell^k_{\text{det}}=u^k_{\text{det}}$ always holds, which is not the case in stochastic settings where $\ell^{k}_{\text{sto}}<u^k_{\text{sto}}$. Moreover, since  $\delta^{k+1}<\dk$ whenever the iteration $k$ is unsuccessful, then $\ell^{k+1}_{\text{det}}>u^k_{\text{det}}$. Likewise, since $\delta^{k+1}>\dk$ on  successful iterations, then $u^{k+1}_{\text{det}}<\ell^k_{\text{det}}$. Given that the equality $\ell^k_{\text{sto}}=u^k_{\text{sto}}$ can not hold in the present stochastic settings, then $\tau$ must be chosen in such a way that at least, both inequalities $\ell^{k+1}_{\text{sto}}>u^k_{\text{sto}}$ and $u^{k+1}_{\text{sto}}<\ell^k_{\text{sto}}$ hold respectively on unsuccessful and successful iterations, analogously to the deterministic framework. This means using~\eqref{successDecrease} and~\eqref{failure}, that $\tau$ must be chosen according to 
\begin{equation}
\rho(\tau\dk)<\frac{\gamma-2}{\gamma+2}\rho(\dk)\quad\text{and}\quad \rho(\tau^{-1}\dk)>\frac{\gamma+2}{\gamma-2}\rho(\dk).\label{tauChoice}
\end{equation}
It follows from~\eqref{tauChoice} that depending on the expression of the forcing function $\rho$, the choice of $\tau$ could depend on $\dk$ and hence should be made at each iteration. Thus, in order to make the present analysis simpler, the following assumption is made.

\begin{assumption}\label{rho}
	The forcing function $\rho\colon(0,+\infty)\to (0,+\infty)$ is such that $\rho(t)=ct^p$, where $c>0$ and $p>1$ are fixed constants.
\end{assumption}

Under Assumption~\ref{rho}, the choice of $\tau$ does not depend on $\dk$. More precisely, it follows from~\eqref{tauChoice} that $\tau$ must be chosen according to $0<\tau^p<\frac{\gamma-2}{\gamma+2}$, for all $k\in\N$, as specified in Algorithm~\ref{algomads}.

\begin{figure*}[ht!]
	\begin{algorithm}[H]
		\caption{SDDS}
		\label{algomads}
		\textbf{[0] Initialization}\\
		\hspace*{10mm}Choose $x^0\in\rn$, $\delta^0>0$, $\ef>0$, $\gamma>2$, $c>0$, $p>1$, $0<\tau<\left(\frac{\gamma-2}{\gamma+2}\right)^{1/p}$, $j_{\max}\in\N$\\ 
		\hspace*{10mm}and $\delta_{\max}=\tau^{-j_{\max}}\delta^0$.\\
		\hspace*{10mm}Set the iteration counter $k \gets 0$.\\
		\textbf{[1] Poll}\\
		\hspace*{10mm}Select a positive spanning set $\mathbb{D}^k$. \\
		\hspace*{10mm}Generate a set $\mathcal{P}^k$ of Poll points such that $\mathcal{P}^k=\{x^k+\dk d:  d\in \mathbb{D}^k\}$.\\ 
		\hspace*{10mm}Obtain estimates $\fok$ and $\fsk$ of $\fd(x^k)$ and $\fd(x^k+s^k)$,  
		respectively, {\blc using objective func-}\\
		\hspace*{10mm}{\blc tion evaluations.}\\
		\hspace*{10mm}\textbf{{\blc Success}}\\
		\hspace*{10mm}If $\fsk -\fok \leq -\gamma c\ef(\dk)^p$ for some $s^k=\dk d^k\in \{\dk d: d\in \mathbb{D}^k\}$,  \\
		\hspace*{10mm}{\blc Set} $x^{k+1}\gets x^k+s^k$, and $\delta^{k+1}\gets \min\{\tau^{-1}\dk,\dmax\}$. \\
		\hspace*{10mm}\textbf{{\blc Failure}}\\
		\hspace*{10mm}Otherwise set $x^{k+1}\gets x^k$ and $\delta^{k+1}\gets\tau\dk$. \\
		\textbf{[2] Termination}\\
		\hspace*{10mm}If no termination criterion is met, \\
		\hspace*{10mm}Set $k\gets k+1$ and go to \textbf{[1]}.\\
		\hspace*{10mm}Otherwise stop.
	\end{algorithm}
	\caption[]{\small{Pseudo code of the Stochastic {\bld D}irectional {\bld D}irect-{\bld S}earch {\bld (SDDS)} algorithm. Success or failure is determined during the Poll at iteration $k$, using information provided by both estimates $\fok$ and $\fsk$ in order to update the step size parameter $\dk$ and the current iterate $x^k$. As long as no stopping criteri{\bld on} {\bld is} met, a new iteration is initiated with a new step size parameter $\delta^{k+1}$.}}
\end{figure*}

\subsection{Probabilistic estimates}\label{ExpSub22}

Following the notation in~\cite{bhattacharya2007basic}, all stochastic quantities in the present manuscript live on the same probability space $(\Omega,\mathcal{F},\pb)$, where $\Omega$ is a nonempty set referred to as the sample space, $\mathcal{F}$ is a collection of events (subsets of $\Omega$) called a $\sigma$-field and $\pb$ is a finite measure on the measurable space $(\Omega,\mathcal{F})$ satisfying $\pr{\Omega}=1$ and referred to as probability measure. The elements $\omega\in\Omega$ are referred to as possible outcomes or sample points. When $\R^n$ is given its Borel $\sigma$-field, i.e., the one generated by the open sets, a random variable or random map $X$ is a measurable map on the probability space $(\Omega,\mathcal{F},\pb)$ into the measurable space $(\R^n,\mathcal{B}(\R^n))$. Measurability meaning that each event $\{X\in I\}:=X^{-1}(I)$ belongs to  $\mathcal{F}$ for all $I\in \mathcal{B}(\R^n)$~\cite{bhattacharya2007basic}. 

The estimates $\fsk$ and $\fok$ constructed at iteration $k$ of Algorithm~\ref{algomads}, based on random information provided by the noisy objective $f_{\Theta}$, can be considered as realizations of random estimates $\Fsk$ and $\Fok$ respectively. Thus, because of the randomness stemming from such random estimates whose behavior influences each iteration $k$, Algorithm~\ref{algomads} { results in a stochastic process} $\{X^k, S^k, \Dk,\Fsk,\Fok \}$. In general, uppercase letters will be used to denote random variables while lowercase letters will be used for their realizations. For example, $x^k=X^k(\omega)$, $s^k=S^k(\omega)$ and $\delta^k=\Delta^k(\omega)$ denote respectively realizations of the random variables $X^k, S^k$ and $\Dk$. Similarly, following the notations in~\cite{audet2019stomads,blanchet2019convergence, chen2018stochastic,paquette2018stochastic},  $\fok=\Fok(\omega)$ and $\fsk=\Fsk(\omega)$ with $\Fok$ and $\Fsk$ denoting respectively estimates of $f(X^k)$ and $f(X^k+S^k)$.

The goal of this work is to show that the stochastic process resulting from Algorithm~\ref{algomads} converges at an appropriate rate with probability one, provided that the sequence $\{(\Fok,\Fsk)\}$ is sufficiently accurate with sufficiently high but fixed probability, conditioned on the past. 

As proposed in~\cite{chen2018stochastic,paquette2018stochastic}, the notion of conditioning on the past is formalized in the following definition similar to those in~\cite{audet2019stomads,blanchet2019convergence,chen2018stochastic,paquette2018stochastic,wang2019stochastic}, where $\mathcal{F}^F_{k-1}$ denotes the $\sigma$-field generated by $F^0_0,F^0_s,F^1_0,F^1_s,$\\$\dots,F^{k-1}_0$ and $F^{k-1}_s$, with $\mathcal{F}^F_{-1}$ being set to equal $\sigma(x^0)$ for completeness. Thus, on can notice that $\E{\Dk|\mathcal{F}^F_{k-1}}=\Dk$ and $\E{X^k|\mathcal{F}^F_{k-1}}=X^k$ for all $k\geq 0$, by construction of the random variables $\Dk$ and $X^k$ in Algorithm~\ref{algomads}.

\begin{definition}\label{probestim2}
	A sequence of random estimates $\{(F^k_0,F^k_s)\}$ is said to be $\beta$-probabilistically $\ef$-accurate with respect to the corresponding sequence $\{X^k, S^k, \Dk\}$ if the events 
	\begin{equation}\label{event1}
	J_k=\{F^k_0, F^k_s,\ \text{are}\ \ef\text{-accurate estimates of}\ \fd(x^k)\ \text{and}\ \fd(x^k+s^k),\ \text{respectively}\}\nonumber
	\end{equation}
	satisfy the following submartingale-like condition
	\begin{equation}\label{beta1}
	{\blc \pr{J_k\ |\ \mathcal{F}^F_{k-1}}}={\blc \E{\ijk\ |\ \mathcal{F}^F_{k-1}}}\geq \beta,\nonumber
	\end{equation}
	where $\ijk$ denotes the indicator function of the event $J_k$, that is $\ijk=1$ if $\omega\in J_k$ and $0$ otherwise.
	
	An estimate is called ``{\it good}'' if $\ijk=1$. Otherwise it is called ``{\it bad}''{\em \cite{audet2019stomads}}.
\end{definition}


Global convergence properties of deterministic directional direct-search methods strongly rely on having the step size parameters approaching zero~\cite{Vicente2013} and the fact that the function value $f(x)$ never increases after an iteration. The main challenge of the analysis in the present stochastic framework lies in the fact that this monotonicity is not always guaranteed. The key to the analysis of Algorithm~\ref{algomads} thus relies on the assumption that accuracy in function estimates ``{\it improves in coordination with the perceived progress of the algorithm}''~\cite{blanchet2019convergence}.  The analysis is based on properties of supermartingales whose increments have a decreasing tendency and depend on the change in objective function values between iterations. 

In order to show that the sequence $\{\Dk\}_{k\in\N}$ of random step size parameters converges to zero with probability one, let make the following key assumption similar to those in~\cite{audet2019stomads,paquette2018stochastic}.

\begin{assumption}\label{keyAssumption}
	For some fixed $\beta\in(0,1)$, and $\ef>0$, the followings hold for the random quantities derived from Algorithm~\ref{algomads}.
	\begin{itemize}
		\item[(i)] The sequence $\{(\Fok,\Fsk)\}$ of estimates is $\beta$-probabilistically $\ef$-accurate.
		\item[(ii)] The sequence $\{(\Fok,\Fsk)\}$ satisfies the following variance condition
		\begin{eqnarray}
		& &\E{\abs{\Fok-f(X^k)}^2|\ \mathcal{F}^F_{k-1}}\leq \ef^2(1-\beta)[\rho(\Dk)]^2 \nonumber\\
		\text{and}  & &\E{\abs{\Fsk-f(X^k+S^k)}^2|\ \mathcal{F}^F_{k-1}}\leq \ef^2(1-\beta)[\rho(\Dk)]^2\label{keysk}
		\end{eqnarray}
	\end{itemize}
\end{assumption}
By means of Assumption~\ref{keyAssumption}-{\it (ii)}, the variance in function estimates is adaptively controlled. Showing therefore that the sequence of random step size parameters converges to zero with probability one, ensures that this variance is driven to zero even though the probability $\beta$ of encountering good estimates remains fixed, thus allowing Algorithm~\ref{algomads} to behave like an exact deterministic method asymptotically.

Moreover, since the estimates satisfying Assumption~\ref{keyAssumption} can easily be constructed using techniques proposed in~\cite{blanchet2019convergence,chen2018stochastic,paquette2018stochastic}, then thorough details about their computations are not provided here again. 
Note however that if $\Theta^0$ and $\Theta^s$ are two independent random variables following the same distribution as $\Theta$ defined in~\eqref{problem}, and if $\Theta^0_i$, $i=1,2,\dots,p^k$ and $\Theta^s_j$, $j=1,2,\dots,p^k$ are independent random samples of $\Theta^0$ and $\Theta^s$ respectively, then the estimates 
\begin{equation}
\Fok=\frac{1}{p^k}\sum_{i=1}^{p^k}f_{\Theta^0_i}(x^k)\quad\text{and}\quad \Fsk=\frac{1}{p^k}\sum_{j=1}^{p^k}f_{\Theta^s_j}(x^k+s^k)      \nonumber
\end{equation}
satisfy Assumption~\ref{keyAssumption} provided that the sample size $p^k$ satisfies \[p^k\geq \frac{V}{\ef^2(1-\sqrt{\beta})[\rho(\dk)]^2},\] where the constant $V>0$ is such that the variance of $f_{\Theta}(x)$ satisfies $\mathbb{V}\left[f_{\Theta}(x)\right] \leq V<+\infty$, for all $x\in\rn$. 

Next is stated a useful lemma  similar to those in~\cite{audet2019stomads,paquette2018stochastic}, linking the probability of obtaining bad estimates to the variance assumption on function values.
\begin{lemma}\label{keyLemma}
	Let Assumption~\ref{keyAssumption} holds. Then for all $k\geq 0$, the followings hold for the random process $\{X^k,\Fok,\Fsk,\Dk\}$ generated by Algorithm~\ref{algomads}
	\begin{eqnarray}
	& &\E{\ijkc\abs{\Fok-f(X^k)}|\ \mathcal{F}^F_{k-1}}\leq \ef(1-\beta)[\rho(\Dk)] \nonumber\\
	\text{and}  & &\E{\ijkc\abs{\Fsk-f(X^k+S^k)}^2|\ \mathcal{F}^F_{k-1}}\leq \ef(1-\beta)[\rho(\Dk)]
	\end{eqnarray}
\end{lemma}
\begin{proof}
	The result is proved using ideas derived from~\cite{audet2019stomads,paquette2018stochastic}. 
	The proof follows straightforwardly from the conditional Cauchy-Schwarz inequality~\cite{bhattacharya2007basic} as follows
	\begin{eqnarray*}
		\E{\ijkc\abs{\Fsk-f(X^k+S^k)}|\ \mathcal{F}^F_{k-1}}&\leq& [\E{\ijkc|\ \mathcal{F}^F_{k-1}}]^{1/2} [\E{\abs{\Fsk-f(X^k+S^k)}^2|\ \mathcal{F}^F_{k-1}}]^{1/2}\\
		&\leq& (1-\beta)^{1/2}\ef(1-\beta)^{1/2}[\rho(\Dk)],
	\end{eqnarray*}
	where the last inequality follows from~\eqref{keysk} and the fact that $\E{\ijkc|\ \mathcal{F}^F_{k-1}}=\pr{\ijkc|\ \mathcal{F}^F_{k-1}}\leq 1-\beta$ thanks to Assumption~\ref{keyAssumption}-{\it (i)}. The proof for $\Fok-f(X^k)$ is the same.
\end{proof} 

\section{A renewal-reward martingale process}\label{ExpSect3}

This section presents a general stochastic process and its associated stopping time $T$ introduced in~\cite{blanchet2019convergence} for the convergence rate analysis of a stochastic trust-region method. It introduces some relevant definition, assumptions and theorem {\bld derived in the analysis of a renewal-reward process in~\cite{blanchet2019convergence}}, that will be useful for the convergence rate analysis presented in Section~\ref{ExpSect4}. Specifically, by considering the stopping time consisting of the time required by SDDS to reach a desired accuracy, Section~\ref{ExpSect4} will aim to show how the properties of this general stochastic process are satisfied for Algorithm~\ref{algomads}. 
Note that some results derived in analyzing this stochastic process in~\cite{blanchet2019convergence} are also used in~\cite{paquette2018stochastic} for the convergence rate analysis of a stochastic line search method. 

\begin{definition}\label{stoppingTime}
	A random variable $T$ is said to be a stopping time with respect to a given discrete time stochastic process $\{X_k\}_{k\in \N}$ if, for each $k\in\N$, the event $\{T=k\}$ belongs to the $\sigma$-field $\sigma(X_1,X_2,\dots,X_k)$ generated by $X_1,X_2,\dots,X_k$.
\end{definition}

Consider a stochastic process $\{(\Phi_k,\Dp)\}_{k\in\N}$ satisfying $\Phi_k\in [0,+\infty)$ and  $\Dp\in [0,+\infty)$ for all $k\in\N$. Define on the same probability space as $\{(\Phi_k,\Dp)\}_{k\in\N}$, a sequence of biased random walk process $\{W_k\}_{k\in\N}$ such that $W_0=1$,
\begin{equation}\label{randomWalk}
\pr{W_{k+1}=1\ |\ \mathcal{F}_k}=q \quad\ \text{and}\quad \pr{W_{k+1}=-1\ |\ \mathcal{F}_k}=1-q, 
\end{equation} 
where $q\in (1/2,1)$ and $\mathcal{F}_k$ denotes the $\sigma$-field generated by $\{(\Phi_0,\Delta^0,W_0), (\Phi_1,\Delta^1,W_1), \dots,$\\$ (\Phi_k,\Delta^k,W_k) \}$. 

Define the following family $\{T_{\epsilon'}\}_{\epr >0}$ of stopping times parameterized by $\epsilon'>0$, with respect to $\{\mathcal{F}_k\}_{k\in\N}$. The following assumptions are made in~\cite{blanchet2019convergence,paquette2018stochastic} in order to derive a bound on $\E{T_{\epsilon'}}$.

\begin{assumption}\label{renewalReward} The following hold for the stochastic process $\{(\Phi_k,\Delta^k,W_k)\}_{k\in\N}$.
	\begin{itemize}
		\item[(i)] There exist constant $\lambda\in (0,+\infty)$ and $\delta_{\max}=\delta^0 e^{\lambda j_{\max}}$, for some integer  $j_{\max}\in \Z$, such that $\Dk\leq \dmx$ for all $k\in\N$.
		\item[(ii)] There exists a constant $\deps=\delta^0 e^{\lambda j_{\epsilon'}}$, for some $j_{\epsilon'}\in\Z, j_{\epsilon'}\leq 0$, such that the following holds for all $k\in\N$,
		\begin{equation}\label{point2}
		\ite\Delta^{k+1}\geq \ite \min\left(\Dp e^{\lambda W_{k+1}}, \deps\right), 
		\end{equation}
		where $W_{k+1}$ satisfies~\eqref{randomWalk} with $q>\frac{1}{2}$.
		\item[(iii)] There exists a nondecreasing function $h:[0,+\infty)\rightarrow (0,+\infty)$ and a constant $\eta>0$ such that 
		\begin{equation}\label{point3}
		\E{\Phi_{k+1}-\Phi_k\ |\ \mathcal{F}_k}\ite\leq -\eta h(\Dp)\ite.
		\end{equation}
	\end{itemize}
\end{assumption}
Note that as highlighted in~\cite{blanchet2019convergence,paquette2018stochastic}, Assumption~\ref{renewalReward} states that conditioned on the past, the nonnegative random sequence $\{\Phi_k\}_{k\in\N}$ decreases by at least $\eta h(\Dp)$ at each iteration provided that $T_{\epsilon'}>k$ and moreover, the sequence $\{\Dp\}_{k\in\N}$ has a tendency to increase whenever it is below some fixed threshold~$\deps$.

The following theorem providing a bound on $\E{T_{\epsilon'}}$ is proved in~\cite{blanchet2019convergence} by observing that the upward drift in the random walk $\{W_k\}_{k\in\N}$ makes the event $\{\Dk\geq\deps \}$ occur sufficiently frequently on average~\cite{blanchet2019convergence,paquette2018stochastic}. Hence, $\E{\Phi_{k+1}-\Phi_k}$ can frequently be bounded by some negative fixed constant, thus leading to a bound on the expected stopping time $\E{T_{\epsilon'}}$.

\begin{theorem}\label{boundRate} Let Assumption~\ref{renewalReward} hold. Then, 
	\begin{equation}
	\E{T_{\epr}}\leq \frac{q}{2q-1}\times \frac{\Phi_0}{\eta h(\deps)}+1 \nonumber
	\end{equation}
\end{theorem}

\section{Convergence rate analysis}\label{ExpSect4}
It follows from Section~\ref{ExpSect3} that Theorem~\ref{boundRate} holds for any stopping time $T_{\epr}$ defined with respect to the filtration $\{\mathcal{F}_k\}_{k\in\N}$, provided that Assumption~\ref{renewalReward} hold for the stochastic process $\{(\Phi_k,\Delta^k,W_k)\}_{k\in\N}$. Thus, the goal of the present section is to show how such a stochastic process satisfying Assumption~\ref{renewalReward} can be constructed in order to bound the expected number of iterations required by Algorithm~\ref{algomads} to achieve $\norminf{\nabla f(X^k)}\leq\epsilon$, for some arbitrary fixed $\epsilon\in (0,1)$, where $\norminf{\cdot}$ denotes the Euclidean norm of $\rn$ as in the remainder of the manuscript. 

\subsection{{\bld Analysis of the stochastic process generated by SDDS}}\label{ExpSub41}
In order to show that Assumption~\ref{renewalReward} holds, 
let impose the following standard assumption on the objective function~$f$.
\begin{assumption}\label{assumpOnf}	
	The function $f$ is bounded from below, i.e., there exists $f_{\min}\in\R$ such that $-\infty < f_{\min}\leq \fd(x)$,$\ $ for all $x\in\rn$. 
\end{assumption}

The following result generalizing that in~\cite{audet2019stomads} provides a bound on the expected decrease in the random function 
\begin{equation}\label{Phik}
\Phi_k:=\frac{\nu}{c\ef}(\fd(X^k)-f_{\min})+(1-\nu)(\Dk)^p.
\end{equation}




\begin{theorem}\label{zerothOrder} 
	Let Assumption~\ref{rho}, \ref{keyAssumption} and~\ref{assumpOnf} hold. Let $\gamma>2, p>1$ and $\tau\in(0,1)$. Let $\nu\in (0,1)$ and $\beta\in (1/2,1)$ be chosen such that 	
	\begin{equation}\label{nuchoice}
	\frac{\nu}{1-\nu}\geq \frac{2(\tau^{-p}-1)}{\gamma-2} \quad\text{and}\quad \frac{\beta}{{1-\beta}}\geq \frac{\nu}{1-\nu}\times\frac{4}{(1-\tau^p)},
	\end{equation}	
	Then the expected decrease in the random function $\Phi_k$ defined in~\eqref{Phik} satisfies
	\begin{equation}\label{expect}
	\E{\Phi_{k+1}-\Phi_k\ |\ \mathcal{F}^F_{k-1}}\leq -\frac{1}{2}\beta(1-\nu)(1-\tau^p)(\Dk)^p.
	\end{equation}
\end{theorem}

\begin{proof}
	The proof is almost identical to that in~\cite{audet2019stomads}, using ideas derived in~\cite{chen2018stochastic,LaBi2016,paquette2018stochastic} and making use of properties of the random function $\Phi_k$ defined in~\eqref{Phik}. It considers two separate cases: good estimates and bad estimates, each of which are broken into whether an iteration is successful or unsuccessful. Define the event $S$ by \[S:=\{\text{The iteration is successful} \}, \] and let $\bar{S}$ denote the complement of $S$.\\
	\textbf{Case 1 (Good estimates, $\ijk=1$)} The overall goal is to show that $\Phi_k$ decreases no matter what type of iteration occurs thus yielding the following bound
	\begin{equation}\label{concl1}
	\E{\ijk(\Phi_{k+1}-\Phi_k)\ |\ \mathcal{F}^F_{k-1}}\leq -\beta(1-\nu)(1-\tau^p)(\Dp)^p.
	\end{equation}
	\begin{itemize}
		\item[(i)] \textit{Successful iteration} $(\isucc=1)$. A decrease occurs in $f$ according to~\eqref{successDecrease} since estimates are good and the iteration is successful, thus implying that
		\begin{equation}
		\ijk\isucc\ \frac{\nu}{c\ef}(\fd(X^{k+1})-\fd(X^k))\leq -\ijk\isucc\nu(\gamma-2)(\Dk)^p.\label{aa1}
		\end{equation}
		The step size parameter is updated according to {\blc $\Delta^{k+1}= \min\{\tau^{-1}\Dp,\dmax\}$.} Hence, 
		\begin{equation}\label{aa2}
		\ijk\isucc(1-\nu)\left[(\Delta^{k+1})^p-(\Dp)^p\right]{\blc \leq}\ijk\isucc(1-\nu)(\tau^{-p}-1)(\Dp)^p.
		\end{equation}
		Then, choosing $\nu$ according to~\eqref{nuchoice} ensures that the right-hand side term  of~\eqref{aa1} dominates that of~\eqref{aa2}, i.e., 
		\begin{equation}
		-\nu(\gamma-2)(\Dp)^p+(1-\nu)(\tau^{-p}-1)(\Dp)^p \leq -\frac{1}{2}\nu(\gamma-2)(\Dp)^p.\label{A1}
		\end{equation}
		Thus, combining~\eqref{aa1}, \eqref{aa2} and~\eqref{A1} yields
		\begin{eqnarray}\label{A3}
		\ijk\isucc (\Phi_{k+1}-\Phi_k)\leq -\ijk\isucc \frac{1}{2}\nu(\gamma-2)(\Dp)^p.
		\end{eqnarray}
		\item[(ii)] \textit{Unsuccessful iteration} $(\isbar=1)$. 
		The step size parameter is decreased while there is a change of zero in function values since the iteration is unsuccessful. Thus,
		\begin{equation}
		\ijk\isbar(\Phi_{k+1}-\Phi_k)= - \ijk\isbar(1-\nu)(1-\tau^p)(\Dp)^p.\label{A5}
		\end{equation}
		
		Then, choosing $\nu$ according to~\eqref{nuchoice} and noticing that $1-\tau^p< \tau^{-p}-1$, ensure that unsuccessful iterations, specifically~\eqref{A5}, provide the worst case decrease when compared to~\eqref{A3}, i.e., the following holds
		\begin{equation}
		-\frac{1}{2}\nu(\gamma-2)(\Dp)^p\leq -(1-\nu)(1-\tau^p)(\Dp)^p.\label{A6}
		\end{equation}
		
		Thus, combining~\eqref{A3}, \eqref{A5}, and~\eqref{A6}, leads to the following bound on the change in $\Phi_k$
		\begin{equation}\label{A9}
		\ijk(\Phi_{k+1}-\Phi_k) = \ijk (\isucc+\isbar)(\Phi_{k+1}-\Phi_k)\leq - \ijk(1-\nu)(1-\tau^p)(\Dp)^p.
		\end{equation}
		Since Assumption~\ref{keyAssumption} holds, then taking conditional expectations with respect to $\mathcal{F}^F_{k-1}$ in both sides of~\eqref{A9} leads to~\eqref{concl1}.
	\end{itemize}
	\textbf{Case 2 (Bad estimates, $\ijkc = 1$).} Since the estimates are bad, an iterate leading to an increase in $f$ and $\Dp$, and hence in $\Phi_k$, can be accepted by Algorithm~\ref{algomads}. Such an increase in $\Phi_k$ is controlled by bounding the variance in function estimates, using~\eqref{keysk}. Then, in order to guarantee that $\Phi_k$ is sufficiently reduced in expectation, the probability of outcome is adjusted to be sufficiently small. The overall goal is to show that 
	\begin{equation}\label{concl2}
	\E{\ijkc(\Phi_{k+1}-\Phi_k)\ |\ \mathcal{F}^F_{k-1}}\leq 2\nu(1-\beta)(\Dp)^p.
	\end{equation}
	
	\begin{itemize}
		\item[(i)] \textit{Successful iteration} $(\isucc=1)$. The change in $f$ is bounded as follows
		\begin{eqnarray}
		\ijkc\isucc \frac{\nu}{c\ef}(\fd(&&\hspace{-1cm}X^{k+1})-\fd(X^k))\nonumber \\
		&\leq& \ijkc\isucc \frac{\nu}{c\ef}\left[(\Fsk -\Fok)+\abs{\fd(X^{k+1})-\Fsk}+\abs{\Fok-\fd(X^{k})}\right]\nonumber \\
		&\leq& \ijkc\isucc \nu \left[-\gamma(\Dp)^p+\frac{1}{c\ef}\left(\abs{\fd(X^{k+1})-\Fsk}+\abs{\Fok-\fd(X^{k})}\right) \right] \label{427}
		\end{eqnarray}
		where the last inequality in~\eqref{427} follows from the fact that $\Fsk -\Fok\leq -\gamma c\ef(\Dp)^p$ for successful iterations. Moreover, as in Case~1, $\Delta^{k+1}= \min\{\tau^{-1}\Dp,\dmax\}$ since the iteration is successful. Thus,
		\begin{equation}\label{aa22}
		\ijkc\isucc(1-\nu)\left[(\Delta^{k+1})^p-(\Dp)^p\right]{\leq}\ijkc\isucc(1-\nu)(\tau^{-p}-1)(\Dp)^p.
		\end{equation}
		Then, choosing $\nu$ according to~\eqref{nuchoice} yields
		\begin{equation}\label{opi}
		-\nu \gamma(\Dp)^p+(1-\nu)(\tau^{-p}-1)(\Dp)^p\leq 0.
		\end{equation}
		Thus, combining~\eqref{427}, \eqref{aa22} and~\eqref{opi} leads to
		\begin{eqnarray}\label{A10}
		\ijkc\isucc (\Phi_{k+1}-\Phi_k)\leq \ijkc\isucc\frac{\nu}{c\ef} (\abs{\fd(X^{k+1})-\Fsk}+\abs{\Fok-\fd(X^{k})})
		\end{eqnarray}
		\item[(ii)] \textit{Unsuccessful iteration} $(\isbar=1)$. $\Dp$ is decreased and the change in function values is zero. Thus, the bound in the change of $\Phi_k$ follows straightforwardly from~\eqref{A5} by replacing $\ijk$ by $\ijkc$. More precisely, the following holds, 
		\begin{eqnarray}
		\ijkc\isbar(\Phi_{k+1}-\Phi_k)&=& - \ijkc\isbar(1-\nu)(1-\tau^p)(\Dp)^p.\nonumber\\
		&\leq&\ijkc\isbar\frac{\nu}{c\ef} (\abs{\fd(X^{k+1})-\Fsk}+\abs{\Fok-\fd(X^{k})})    \label{A12}
		\end{eqnarray}
		Then, combining~\eqref{A10} and~\eqref{A12}, yields
		\begin{eqnarray}\label{A16}
		\ijkc (\Phi_{k+1}-\Phi_k)\leq \ijkc\frac{\nu}{c\ef}(\abs{\fd(X^{k+1})-\Fsk}+\abs{\Fok-\fd(X^{k})}).
		\end{eqnarray}
		Taking conditional expectations with respect to $\mathcal{F}^F_{k-1}$ in both sides of~\eqref{A16} and applying Lemma~\ref{keyLemma} leads to~\eqref{concl2}.
	\end{itemize}
	Now, combining expectations~\eqref{concl1} and~\eqref{concl2} leads to
	\begin{eqnarray}
	\E{\Phi_{k+1}-\Phi_k\ |\ \mathcal{F}^{F}_{k-1}}&=& \E{(\ijk+\ijkc)(\Phi_{k+1}-\Phi_k)\ |\ \mathcal{F}^{F}_{k-1}}\nonumber \\
	&\leq& \left[-\beta(1-\nu)(1-\tau^p)+ 2\nu (1-\beta)\right](\Dp)^p.\label{C1}
	\end{eqnarray}
	Then, choosing $\beta$ according to~\eqref{nuchoice} ensures that
	\begin{equation}\label{C2}
	-\beta(1-\nu)(1-\tau^p)+ 2\nu(1-\beta)\leq -\frac{1}{2}\beta(1-\nu)(1-\tau^p).
	\end{equation}
	Hence, \eqref{expect} follows from~\eqref{C1} and~\eqref{C2}, which achieves the proof.
\end{proof}

Summing both sides of~\eqref{expect} over $k\in\N$ and taking expectations with respect to $\mathcal{F}^F_{k-1}$ lead to the following result generalizing that in~\cite{audet2019stomads}, which shows in particular that the sequence $\{\Dk\}_{k\in\N}$ of step size parameters converges to zero with probability one.
\begin{theorem}
	Let all assumptions that were made in Theorem~\ref{zerothOrder} hold. Then, the sequence $\{\Dk\}_{k\in\N}$ of step size parameters generated by Algorithm~\ref{algomads} satisfies for $p>1$,
	\begin{equation}
	\sum_{k=0}^{+\infty}(\Dk)^p<+\infty\quad\text{almost surely.}  \nonumber
	\end{equation}
\end{theorem}

Consider the  stochastic process $\{(\Phi_k,\Dp,W_k)\}_{k\in\N}$, where $\Phi_k$ is the same random function in Theorem~\ref{zerothOrder}, $\Dp$ is the random step size parameter and $W_k=2(\ijk-\frac{1}{2})$. Define $\hat{p}=\min(p-1,1)$ for some fixed $p>1$. For some arbitrary fixed $\epr\in (0,1)$, consider the following random time $T_{\epr}$ defined by

\begin{equation}\label{stopGrad}
T_{\epr}=\inf \left\lbrace k\in\N:\norminf{\nabla f(X^k)}^{1/\hat{p}}\leq \epr \right\rbrace
\end{equation}
Then, $T_\epr$ is a stopping time for the stochastic process generated by Algorithm~\ref{algomads} and is consequently a stopping time for  $\{(\Phi_k,\Dp,W_k)\}_{k\in\N}$~\cite{blanchet2019convergence,paquette2018stochastic}. Moreover, $T_{\epsilon^{1/\hat{p}}}$ is the number of iterations required by Algorithm~\ref{algomads} to drive the norm of the gradient of $f$ below $\epsilon\in (0,1)$. This latter remark will help to derive the main result of the present work in Theorem~\ref{mainResult}. 

In order to apply Theorem~\ref{boundRate} to $T_{\epr}$, the remainder of this section is devoted to showing that Assumption~\ref{renewalReward} holds for the previous stochastic process. First, notice that since Theorem~\ref{zerothOrder} holds without using any information about the existence of the gradient of $f$, then by multiplying both sides of~\eqref{expect} by $\ite$, Assumption~\ref{renewalReward}-{\it (iii)}, trivially holds with $\eta=\frac{1}{2}\beta(1-\nu)(1-\tau^p)$ and $h(x)=x^p$. By choosing $\lambda$ such that $e^{\lambda}=\tau^{-1}$ and noticing that $\Dp\leq \dmax=\delta^0 e^{\lambda j_{\max}}$ in Algorithm~\ref{algomads} for all $k\in\N$, then Assumption~\ref{renewalReward}-{\it (i)}  holds.

Then, before showing that Assumption~\ref{renewalReward}-{\it (ii)} also holds, let emphasize that as in the deterministic framework, polling directions in Algorithm~\ref{algomads} are chosen in such a way that their significant deterioration can be avoided asymptotically, i.e., in such a way to ensure that they never become close to loosing the positive spanning property~\cite{Vicente2013}. For this purpose, let recall the following definition of the {\it cosine measure}~\cite{CoScVibook,KoLeTo03a} of a positive spanning set $\mathbb{D}^k$ with non-zero vectors
\begin{equation}
\kappa(\mathbb{D}^k):=\underset{v\in\rn}{\min}\  \underset{d\in\mathbb{D}^k}{\max}\frac{v^{\top}d}{\norminf{v}\norminf{d}} \nonumber
\end{equation}
In order to avoid the aforementioned deterioration of polling directions, the positive spanning sets are required to satisfy the following assumption~\cite{KoLeTo03a,Vicente2013} where the size of the directions does not tend to infinite or approach zero, and the cosine measure always stays positive.
\begin{assumption}\label{koleto}
	The followings hold for all positive spanning sets $\mathbb{D}^k$ used for polling in Algorithm~\ref{algomads}. There exists a constant $\kappa_{\min}>0$ such that $\kappa(\mathbb{D}^k)>\kappa_{\min}\ $ for all $k$. There exist constants $d_{\min}>0$ and $d_{\max}>0$ such that $d_{\min}\leq\norminf{d}\leq d_{\max}$ for all $d\in \mathbb{D}^k$.
\end{assumption}
The following result from~\cite{KoLeTo03a} will be useful for the remaining of the analysis, and specifically the proof of the key result in Lemma~\ref{implication}. It shows by means of the cosine measure $\kappa(\mathbb{D}^k)$, how far can be in the worst case, the steepest descent direction, from the vector in $\mathbb{D}^k$ which makes the smallest angle with $v=-\nabla f(x^k)$. This means in term of descent that, there exists $d^k_*\in\mathbb{D}^k$ such that 
\begin{equation}\label{detoil}
\kappa(\mathbb{D}^k) \norminf{\nabla f(x^k)} \norminf{d^k_*} \leq -\nabla f(x^k)^{\top}d^k_*.
\end{equation} 

For the remaining of the analysis, 
the following standard assumption is also imposed on the gradient of $f$. 
\begin{assumption}\label{gradient}
	The gradient $\nabla f$ of the objective function $f$ is $L$-Lipschitz continuous everywhere.
\end{assumption} 
Then, define the constant $\deps$ as follows
\begin{equation}\label{seven}
\deps =\frac{\epr}{\zeta} \quad\text{with}\quad \zeta > \left[\kappa_{\min}^{-1}\left(Ld_{\max}+(\gamma+2)c\ef d_{\min}^{-1}\right)\right]^{1/\hat{p}},
\end{equation}
where without loss of generality, $Ld_{\max}>\kappa_{\min}$ so that $\deps<1$ for the needs of the analysis and specifically, the proof of Lemma~\ref{implication}. Then following~\cite{blanchet2019convergence}, it can be assumed without any loss of generality that $\deps=\tau^{-i}\delta^0$, for some integer $i\leq 0$. Hence, for any $k$, $\Dp=\tau^{i_k}\deps$, for some integer~$i_k$. Thus, what remains to be proved in Assumption~\ref{renewalReward} in order to apply Theorem~\ref{boundRate} is the dynamics~\eqref{point2}.
Note however that the proof of the dynamics~\eqref{point2} which will be achieved in Lemma~\ref{dynamics3}, need the following intermediate key result. 
Indeed, in the stochastic trust region framework of~\cite{blanchet2019convergence}, the proof of a similar dynamics strongly relies on the fact that any iteration $k$, where $\norme{\nabla f(x_k)}>\epsilon$ and for which ``{\it good}'' model and estimates occur, is successful provided that the {\it trust region radius} $\delta_k$ is bellow a threshold $\Delta_{\epsilon}$. Nevertheless, unlike the trust region framework where informations can possibly easily be derived on the true gradient $\nabla f(x^k)$ using those provided by the {\it gradient estimate} $g_k$, the algorithmic framework of the present work does not use any gradient information. Thus, the main challenge in proving that Assumption~\ref{renewalReward}-{\it (ii)} holds, lies in linking the event $\accolade{\norminf{\nabla f(X^k)}^{1/\hat{p}}>\epr}$ to a successful iteration of Algorithm~\ref{algomads}, which is done next.

\begin{lemma}\label{implication} Assume that Assumption~\ref{gradient} and~\ref{koleto} hold and that $\dpl\leq\deps$. Let $f_0^k$ and $f_s^k$ be $\ef$-accurate estimates of $f(x^k)$ and $f(x^k+s^k)$ respectively. If $\ \norminf{\nabla f(x^k)}^{1/\hat{p}}>\epr$, then 
	\begin{equation}\label{firstbound}
	\fsk-\fok\leq -\gamma c\ef(\dpl)^p.\nonumber
	\end{equation}
	In particular, this means that the iteration $k$ of Algorithm~\ref{algomads} is successful.
\end{lemma}

\begin{proof} The proof uses elements derived in~\cite{KoLeTo03a}. Suppose that $\dpl\leq\deps$ and assume in contradiction that $\fsk-\fok>-\gamma c\ef(\dpl)^p$. Since the estimates $f_0^k$ and $f_s^k$ are $\ef$-accurate, then it follows from the following equality
	\begin{equation}
	f(x^k+s^k)-f(x^k)=f(x^k+s^k)-\fsk +(\fsk-\fok)+\fok-f(x^k) \nonumber
	\end{equation}
	that
	\begin{equation}\label{sup}
	f(x^k+s^k)-f(x^k)+(\gamma+2)c\ef(\dpl)^p\geq 0.
	\end{equation}
	Recall that $s^k=\dk d$ where $d\in\mathbb{D}^k$ denotes any direction used by Algorithm~\ref{algomads} at iteration $k$. It follows from the mean value theorem, combined with~\eqref{sup}, that there exists a constant $\mu_k\in\left[0,1\right]$ such that 
	\begin{equation}\label{meanIneq}
	0\leq \dk \nabla f(x^k+\mu_k\dk d^k_*)^{\top}d^k_*+(\gamma+2)c\ef (\dpl)^p,
	\end{equation}
	where $d^k_*$ is the direction satisfying~\eqref{detoil}. Dividing both sides of~\eqref{meanIneq} by $\dk$ and subtracting $\nabla f(x^k)^{\top}d^k_*$, yields 
	\begin{equation}\label{mean2}
	-\nabla f(x^k)^{\top}d^k_*\leq \left[\nabla f(x^k+\mu_k\dk d^k_*)- \nabla f(x^k)\right]^{\top}d^k_*+(\gamma+2)c\ef (\dpl)^{p-1}.
	\end{equation}
	Putting~\eqref{detoil} and~\eqref{mean2} together, yields 
	\begin{equation}\label{mean3}
	\kappa(\mathbb{D}^k) \norminf{\nabla f(x^k)} \norminf{d^k_*}\leq \left[\nabla f(x^k+\mu_k\dk d^k_*)- \nabla f(x^k)\right]^{\top}d^k_*+(\gamma+2)c\ef (\dpl)^{p-1}.
	\end{equation}
	Then, dividing both sides of~\eqref{mean3} by $\kappa(\mathbb{D}^k)\norminf{d^k_*}$ and using Assumption~\ref{gradient} and~\ref{koleto}, lead to
	\begin{eqnarray}
	\norminf{\nabla f(x^k)} &\leq& \kappa_{\min}^{-1}\left[Ld_{\max}\dk+(\gamma+2)c\ef d_{\min}^{-1}(\dpl)^{p-1} \right]\nonumber \\
	&\leq& \kappa_{\min}^{-1}\left(Ld_{\max}+(\gamma+2)c\ef d_{\min}^{-1}\right)(\dpl)^{\min (p-1,1)},\label{last}
	\end{eqnarray}
	where the inequality~\eqref{last} follows from the fact that $\dpl\leq\deps<1$.
	Now, recall that $\hat{p}=\min (p-1,1)$ and let $L_1:=\kappa_{\min}^{-1}\left(Ld_{\max}+(\gamma+2)c\ef d_{\min}^{-1}\right)$. Then, it follows from~\eqref{last} that 
	\begin{equation}\label{lst}
	\norminf{\nabla f(x^k)}^{1/\hat{p}}\leq L_1^{1/\hat{p}}\dk\leq L_1^{1/\hat{p}}\deps=L_1^{1/\hat{p}}\frac{\epr}{\zeta}\leq \epr, 
	\end{equation}
	where the last inequality in~\eqref{lst} follows from~\eqref{seven}, which achieves the proof.
\end{proof}
Finally, the following result shows that the dynamics~\eqref{point2} of Assumption~\ref{renewalReward}-{\it (ii)} holds.

\begin{lemma}\label{dynamics3}
	Let Assumption~\ref{gradient} and all assumptions that were made in Theorem~\ref{zerothOrder} hold. Then Assumption~\ref{renewalReward}-{\it (ii)} is satisfied for the random variable $W_k=2(\ijk-\frac{1}{2})$, $\lambda=-\ln(\tau)$ and $q=\beta$.
\end{lemma}

\begin{proof}
	The result is proved by adapting the proof of a similar Lemma  from~\cite{blanchet2019convergence}. First, notice that~\eqref{point2} trivially holds when $\ite=0$. Thus, the remaining of the proof is devoted to showing that conditioned on the event $\{T_{\epsilon'}>k\}$, i.e., when $\ite=1$, then the following holds
	\begin{equation}
	\Delta^{k+1}\geq\min\left\lbrace \deps, \min\left\lbrace \tau^{-1}\Dp, \dmax \right\rbrace \ijk+\tau\Dp\ijkc \right\rbrace.
	\end{equation}
	Notice that every realization such that $\dpl>\deps$ also satisfies $\dpl\geq\tau^{-1}\deps$ whence $\delta^{k+1}\geq\tau\dpl\geq\deps$. Now, assume that $\dpl\leq\deps$. Since $T_{\epsilon'}>k$, then it is the case that $\norminf{\nabla f(x^k)}^{1/\hat{p}}>\epsilon'$. If $\ijk=1$, then the estimates are good and are specifically $\ef$-accurate. Hence, it follows from Lemma~\ref{implication} that the $k$th iteration is successful. Thus, $x^{k+1}=x^k+s^k$ and $\delta^{k+1}=\min\left\lbrace \tau^{-1}\dpl,\dmax \right\rbrace$. But if $\ijk=0$, i.e., $\ijkc=1$, then the inequality $\delta^{k+1}\geq \tau\dpl$  always holds by the dynamics of Algorithm~\ref{algomads}. The proof is complete by noticing finally that $\pr{J_k|\mathcal{F}_{k-1}^F}\geq q=\beta$.
\end{proof}

\subsection{Complexity result and first-order optimality conditions}\label{ExpSub42}

The following result provides a bound on the expected number of iterations taken by Algorithm~\ref{algomads} before $\left\lbrace \norminf{\nabla f(X^k)}\leq\epsilon \right\rbrace$ occurs and is the main result of the present work.

\begin{theorem}\label{mainResult}
	Let Assumption~\ref{gradient} and all assumptions that were made in Theorem~\ref{zerothOrder} hold with $\beta\in(1/2,1)$ and $\nu\in(0,1)$ satisfying~\eqref{nuchoice}. Consider Algorithm~\ref{algomads} and the corresponding stochastic process. For some arbitrary fixed $\epsilon\in (0,1)$, consider the random time $T^{\star}_{\epsilon}$ defined by
	\begin{equation}\label{stopeps}
	T^{\star}_{\epsilon}=\inf \left\lbrace k\in\N:\norminf{\nabla f(X^k)}\leq \epsilon \right\rbrace.
	\end{equation}
	Then, 
	\begin{equation}\label{mResult}
	\E{T^{\star}_{\epsilon}}\leq \frac{2\Phi_0 L_2}{(2\beta-1)(1-\nu)(1-\tau^p)}\epsilon^{\frac{-p}{\min(p-1,1)}}+1,
	\end{equation}
	where $L_2:=\left[1+ \kappa_{\min}^{-1}\left(Ld_{\max}+(\gamma+2)c\ef d_{\min}^{-1}\right)\right]^{\frac{p}{\min(p-1,1)}}$.
	i.e., the expected number of iterations taken by Algorithm~\ref{algomads} to reduce the gradient below $\epsilon\in (0,1)$ is bounded in $\mathcal{O}\left(\epsilon^{\frac{-p}{\min(p-1,1)}}/(2\beta-1)\right)$.
\end{theorem}
\begin{proof}
	As shown previously, since Assumption~\ref{renewalReward} holds for the stochastic process $\{(\Phi_k,\Dp,W_k)\}_{k\in\N}$ generated by Algorithm~\ref{algomads}, with $q=\beta$, $h(x)=x^p$, $\eta=\frac{1}{2}\beta(1-\nu)(1-\tau^p)$ and $\deps=\epr/\zeta$, then Theorem~\ref{boundRate} applies for the stopping time $T_\epr$ defined in~\eqref{stopGrad}. Thus, the following inequality holds for all $\epr\in (0,1)$
	\begin{equation}\label{hhj}
	\E{T_{\epr}}\leq \frac{\beta}{2\beta-1}\times \frac{\Phi_0 \zeta^p}{\eta \epr^p}+1,
	\end{equation}
	where $\zeta^p > \left[\kappa_{\min}^{-1}\left(Ld_{\max}+(\gamma+2)c\ef d_{\min}^{-1}\right)\right]^{p/\hat{p}}$ thanks to~\eqref{seven}, with $\hat{p}=\min(p-1,1)$. Now, let $\epsilon\in(0,1)$ be arbitrary fixed. Then, $\epsilon^{1/\hat{p}}\in (0,1)$, which means that~\eqref{hhj} holds in particular for $\epr=\epsilon^{1/\hat{p}}$. By noticing moreover that $T_{\epsilon^{1/\hat{p}}}=T^{\star}_{\epsilon}$, then it follows from~\eqref{hhj} that 
	\begin{equation}\label{final}
	\E{T^{\star}_{\epsilon}}\leq \frac{\beta}{2\beta-1}\times \frac{2\Phi_0 \zeta^p}{\beta(1-\nu)(1-\tau^p)}\epsilon^{-p/\hat{p}}+1.
	\end{equation}
	Since $p/\hat{p}=p/\min(p-1,1)\geq 2$ for all $p>1$, then~\eqref{mResult} results from~\eqref{final} by choosing $\zeta^p$ according to $\zeta^p=\left[1+ \kappa_{\min}^{-1}\left(Ld_{\max}+(\gamma+2)c\ef d_{\min}^{-1}\right)\right]^{p/\hat{p}}$, which achieves the proof.
\end{proof}

The following $\liminf$-type first-order necessary optimality condition is a simple consequence of the complexity result of Theorem~\ref{mainResult}. It shows the existence of a subsequence of random iterates generated by Algorithm~\ref{algomads} which drives the norm of the gradient of $f$ to zero with probability one. Note that a similar corollary has been derived in~\cite{paquette2018stochastic}.
\begin{theorem}\label{liminf}
	Let Assumption~\ref{gradient} and all assumptions that were made in Theorem~\ref{zerothOrder} hold. Then the sequence $\{X^k\}_{k\in\N}$ of random iterates generated by Algorithm~\ref{algomads} satisfies
	\begin{equation}\label{liminfGrad}
	\underset{k\to+\infty}{\liminf}\norminf{\nabla f\left(X^k\right)}=0\quad\text{almost surely}.
	\end{equation}
\end{theorem}

\section*{\blr Discussion}
This manuscript presents the first convergence rate analysis of a broad class of stochastic directional direct-search (SDDS) algorithms, designed for the unconstrained optimization of noisy blackboxes, and based on imposing a sufficient decrease condition when accepting new iterates. Using an existing supermartingale-based framework for the analysis, the methodology for deriving the worst case complexity of SDDS algorithms heavily relies on bounding an expected stopping time associated to the stochastic process generated by the algorithms. The analysis showed that SDDS algorithms have the same worst case complexity as any other first-order optimization method in a nonconvex setting. In particular, this complexity bound matches in some sense its deterministic counterparts despite the fact that function estimates are sometimes allowed to be arbitrarily inaccurate. The main novelty of the present research compared to many others on the worst case complexity analysis of stochastic DFO methods, lies in the fact that the proposed method does not need any gradient information to find descent directions. 

The analysis in the present manuscript strongly relies on the assumption that function estimates are unbiased. Thus, obtaining worst case complexity results when such estimates are possibly biased is a topic for future research. 

\section*{Acknowledgments}
The author is grateful to SÈbastien Le~Digabel and Charles Audet from Polytechnique Montr\'eal and Michael Kokkolaras from McGill university for valuable discussions and constructive suggestions that improve the quality of the presentation. This work is supported by the NSERC CRD RDCPJ 490744-15 grant and by an Innov\'E\'E grant, both in collaboration with Hydro-Qu\'ebec and Rio Tinto. 


\bibliographystyle{plain}
\bibliography{bibliography,bib}

\begin{thebibliography}{10}

\bibitem{alarie2019optimization}
S.~Alarie, C.~Audet, P.~Y. Bouchet, and S.~Le Digabel.
\newblock {Optimization of noisy blackboxes with adaptive precision}.
\newblock Technical Report G-2019-84, Les cahiers du GERAD, 2019.

\bibitem{AuDe03a}
C.~Audet and J.E. {Dennis, Jr.}
\newblock Analysis of generalized pattern searches.
\newblock {\em SIAM Journal on Optimization}, 13(3):889--903, 2003.

\bibitem{AuDe2006}
C.~Audet and J.E. {Dennis, Jr.}
\newblock {Mesh Adaptive Direct Search Algorithms for Constrained
  Optimization}.
\newblock {\em SIAM Journal on Optimization}, 17(1):188--217, 2006.

\bibitem{audet2019stomads}
C.~Audet, K.~J. Dzahini, M.~Kokkolaras, and S.~Le Digabel.
\newblock {StoMADS: Stochastic blackbox optimization using probabilistic
  estimates}.
\newblock Technical Report G-2019-30, Les cahiers du GERAD, 2019.

\bibitem{AuHa2017}
C.~Audet and W.~Hare.
\newblock {\em {Derivative-Free and Blackbox Optimization}}.
\newblock Springer Series in Operations Research and Financial Engineering.
  Springer International Publishing, Cham, Switzerland, 2017.

\bibitem{berahas2019global}
A.~S. Berahas, L.~Cao, and K.~Scheinberg.
\newblock {Global Convergence Rate Analysis of a Generic Line Search Algorithm
  with Noise}.
\newblock {\em arXiv}, 2019.

\bibitem{bhattacharya2007basic}
R.N. Bhattacharya and E.C. Waymire.
\newblock {\em {A basic course in probability theory}}, volume~69.
\newblock Springer, 2007.

\bibitem{blanchet2019convergence}
J.~Blanchet, C.~Cartis, M.~Menickelly, and K.~Scheinberg.
\newblock {Convergence Rate Analysis of a Stochastic Trust-Region Method via
  Supermartingales}.
\newblock {\em INFORMS Journal on Optimization}, 2019.

\bibitem{chen2018stochastic}
R.~Chen, M.~Menickelly, and K.~Scheinberg.
\newblock {Stochastic optimization using a trust-region method and random
  models}.
\newblock {\em Mathematical Programming}, 169(2):447--487, 2018.

\bibitem{CoScVibook}
A.R. Conn, K.~Scheinberg, and L.N. Vicente.
\newblock {\em {Introduction to Derivative-Free Optimization}}.
\newblock MOS-SIAM Series on Optimization. SIAM, Philadelphia, 2009.

\bibitem{GrRoViZh2015}
S.~Gratton, C.W. Royer, L.N. Vicente, and Z.~Zhang.
\newblock {Direct search based on probabilistic descent}.
\newblock {\em SIAM Journal on Optimization}, 25(3):1515--1541, 2015.

\bibitem{KoLeTo03a}
T.G. Kolda, R.M. Lewis, and V.~Torczon.
\newblock Optimization by direct search: New perspectives on some classical and
  modern methods.
\newblock {\em SIAM Review}, 45(3):385--482, 2003.

\bibitem{LaBi2016}
J.~Larson and S.C. Billups.
\newblock {Stochastic derivative-free optimization using a trust region
  framework}.
\newblock {\em Computational Optimization and Applications}, 64(3):619--645,
  2016.

\bibitem{paquette2018stochastic}
C.~Paquette and K.~Scheinberg.
\newblock {A Stochastic Line Search Method with Expected Complexity Analysis}.
\newblock {\em SIAM Journal on Optimization}, 30(1):349--376, 2020.

\bibitem{Vicente2013}
L.~N. Vicente.
\newblock Worst case complexity of direct search.
\newblock {\em EURO Journal on Computational Optimization}, 1(1):143--153,
  2013.

\bibitem{wang2019stochastic}
X.~Wang and Y.~Yuan.
\newblock {Stochastic Trust Region Methods with Trust Region Radius Depending
  on Probabilistic Models}.
\newblock {\em arXiv}, 2019.

\end{thebibliography}

\end{document}